\documentclass[]{interact}

\usepackage{amsmath,amsthm,amsfonts,amssymb,bm}

\usepackage[utf8]{inputenc}

\usepackage[table,dvipsnames]{xcolor}
\usepackage{hhline}
\usepackage[labelfont=bf, textfont = footnotesize]{caption}
\usepackage{stmaryrd}
\usepackage{caption}
\usepackage{enumerate}
\usepackage{hyperref}
\usepackage{bbm}
\usepackage{ushort}
\usepackage{dsfont}
\RequirePackage{booktabs}
\usepackage{comment}
\usepackage{mathabx}
\usepackage{soul}
\usepackage{enumitem}

\usepackage[numbers,sort&compress]{natbib}
\bibpunct[, ]{[}{]}{,}{n}{,}{,}
\renewcommand\bibfont{\fontsize{10}{12}\selectfont}
\makeatletter
\def\NAT@def@citea{\def\@citea{\NAT@separator}}
\makeatother

\usepackage{graphicx}

\newcommand{\R}{\ensuremath{\mathds{R}}}
\newcommand{\C}{\ensuremath{\mathds{C}}}
\newcommand{\N}{\ensuremath{\mathds{N}}}

\newcommand{\Sphere}{\ensuremath{\mathds{S}}}

\newcommand{\Lp}[1]{\ensuremath{\mathds{L}}^{#1}} 
\newcommand{\Lpn}[3]{\ensuremath{\| #1\|_{\mathds{L}^{#2}(#3)}}}
\newcommand{\Lpnlr}[3]{\ensuremath{\left\| #1\right\|_{\mathds{L}^{#2}\left(#3\right)}}}
\newcommand{\Lpinfty}[1]{\ensuremath{\left\| #1 \right\|_{\infty}}}
\newcommand{\Sobolev}[2]{\ensuremath{\mathcal{W}^{#1}\left(#2\right)}} 

\newcommand{\Lebesgue}[2]{\thickspace\mathrm{d}\lambda^{#1}(#2)} 
\newcommand{\inte}{\thickspace\mathrm{d}} 
\newcommand{\surfacemeasure}[2]{\thickspace\mathrm{d}\sigma^{#1}(#2)} 
\newcommand{\indikator}{\mathds{1}} 
\newcommand{\sets}[2]{\ensuremath{\left\{ #1 \thinspace\left \vert \thinspace #2\right.\right\}}} 
\newcommand{\scalarpr}[2]{\ensuremath{\left\langle #1, #2 \right\rangle}} 

\newcommand{\ball}[1]{\ensuremath{\mathrm{B}_{#1}}} 
\newcommand{\ballc}[1]{\ensuremath{\mathrm{B}^c_{#1}}} 
\newcommand{\Mn}{\ensuremath{\mathcal{M}_n}} 

\newcommand{\argmin}{\mathop{\mathrm{arg\,min}}}

\newcommand{\EW}[1]{\ensuremath{\mathds{E}\left[#1\right]}} 
\newcommand{\EWind}[2]{\ensuremath{\mathds{E}_{#1}[#2]}} 
\newcommand{\EWindlr}[2]{\ensuremath{\mathds{E}_{#1}\left[#2\right]}} 
\newcommand{\Vari}[1]{\ensuremath{\mathds{V}\hspace{-2pt}\text{ar}\left(#1\right)}} 
\newcommand{\Varind}[2]{\ensuremath{\mathds{V}\hspace{-2pt}\text{ar}_{#1}(#2)}} 
\newcommand{\PP}[1]{\mathds{P}\left(#1\right)} 

\newcommand{\PPindlr}[2]{\mathds{P}_{#1}\left(#2\right)}
\newcommand{\KL}[2]{\ensuremath{\mathrm{KL}\left(#1, #2\right)}} 

\newcommand{\radon}[1]{\ensuremath{\mathcal{R}\left[#1\right]}} 
\newcommand{\radonind}[2]{\ensuremath{\mathcal{R}_{#1}\left[#2\right]}} 
\newcommand{\Fourier}[2]{\ensuremath{\mathcal{F}_{#1}\left[#2\right]}} 
\newcommand{\invFourier}[2]{\ensuremath{\mathcal{F}^{-1}_{#1}\left[#2\right]}} 

\newtheorem{thm}{Theorem}[section]
\newtheorem{lem}[thm]{Lemma}
\newtheorem{cor}[thm]{Corollary}
\newtheorem{prop}[thm]{Proposition}

\theoremstyle{definition}

\newtheorem{rem}[thm]{Remark}
\newtheorem{ex}[thm]{Example}

\numberwithin{equation}{section}

\begin{document}

\title{Adaptive local density estimation in tomography}

\author{
\name{Sergio Brenner Miguel\textsuperscript{a} and Janine Steck\textsuperscript{b}}
\affil{\textsuperscript{a}Institut f\"{u}r angewandte Mathematik und Interdisciplinary Center for Scientific Computing (IWR), Im Neuenheimer Feld 205, Heidelberg University, Germany, brennermiguel\raisebox{-1.5pt}{@}math.uni-heidelberg.de}
\affil{\textsuperscript{b}Institute of Mathematics, Humboldt-Universit\"{a}t zu Berlin, janine.steck\raisebox{-1.5pt}{@}hu-berlin.de }
}

\maketitle

\begin{abstract}
	We study the non-parametric estimation of a multidimensional unknown density $f$ in a tomography problem based on independent and identically distributed observations, whose common density is proportional to the Radon transform of $f$. We identify the underlying statistical inverse problem and use a spectral cut-off regularisation to deduce an estimator. A fully data-driven choice of the cut-off parameter $m \in \R^+$ is proposed and studied. To discuss the bias-variance trade off, we consider Sobolev spaces and show the minimax-optimality of the spectral cut-off density estimator. In a simulation study, we illustrate a reasonable behaviour of the studied fully data-driven estimator.
\end{abstract}


\section{Introduction}\label{sec_intro}

\indent In this work, we consider the estimation of an unknown density $f: \R^d \rightarrow \R_+$, $d\in\N$ with $d\geq2$, in a positron emission tomography (PET) problem. More precisely, we only have access to observations $\left(S_i, U_i\right)$, $i=1,\ldots,n$, with common density $\rho_d^{-1}\radon{f}$, where $\radon{f}$ denotes the Radon transform of $f$ and $\rho_d := (2 \pi)^{d/2}/\Gamma\left(d/2\right)$ with $\Gamma(\cdot)$ denotes the $\Gamma$-function. Thus, the estimation of $f$ given the observations $\left(S_i, U_i\right)$, $i=1,\ldots,n$, can be interpreted as a statistical inverse problem.\\
Next to the PET, the computerised tomography (CT) is an example of tomography related to statistical inverse problems, as explained in Bissantz et al. \cite{bissantz2014}. While the PET problem can be identified as a density-type inverse problem, the CT can be seen as a regression-type inverse problem. 
In both cases, the Radon transform occurs naturally. For a detailed study of the Radon transform and its properties, we refer to Natterer \cite{Natterer1986}. \\
In a CT-type problem with fixed design, Natterer \cite{natterer1980} investigates the reconstruction of the function $f$  over Sobolev spaces, while a random design is considered in Korostelev and Tsybakov \cite{korostelev1993}. A generalisation of the Radon transform, the exponential Radon transform, has been used in Abhishek \cite{Abhishek2023} to construct a kernel-type estimator and proves its minimax-optimality. Further, in Arya and Abhishek \cite{AryaAbhishek2022} a data-driven choice of the upcoming bandwidth is studied.
\\
For the PET setting, we mainly refer to Cavalier \cite{cavalier2000efficient}, where a kernel-type density estimator is proposed and  minimax-optimality is shown. Similar to the PET, Butucea et al. \cite{butucea2007} analyse a quantum tomography version using Wigner functions. Here, the authors consider a PET model with noisy observations of $U_i$, $i=1,\ldots,n$. They also construct a kernel estimator for the Wigner function and prove the minimax-optimality. Additionally, they investigate data-driven choices of the upcoming smoothing parameters.
In this work, we revisit the model and kernel-type estimator considered in Cavalier \cite{cavalier2000efficient}. Our main contributions are the proof of the minimax-optimality of the kernel-type estimator with respect to a different regularity class, the classical Sobolev ellipsoids, and the study of the data-driven choice of the bandwidth. Inspired by Brenner Miguel et al. \cite{Brenner-MiguelComteJohannes2023}, we prove both, an upper and a lower bound of the mean squared error using results from Tsybakov \cite{Tsybakov2008}. In addition, we propose a fully data-driven choice of the upcoming bandwidth of the kernel-type estimator, which faces a logarithmic deterioration compared to the minimax-optimal rate. Thereby, the proof for the upper bound adapts the structure from \cite{Comte2017} and for the lower bound from \cite{Jirak_2014}.\\
The paper is organised as follows. In this section, we continue with an introduction of the Radon transform. Based on the observations $\left(S_i, U_i\right)$, $i=1,\ldots,n$, we introduce the kernel estimator based on the spectral cut-off approach. Section \ref{sec_mm} states an upper bound of the minimax rate over the Sobolev ellipsoids, followed by the corresponding lower bound. In Section \ref{sec_dd} we propose a fully data-driven procedure for the choice of the spectral cut-off parameter. Finally, Section \ref{sec_nu} aims to visualise the results obtained in the previous section using the statistical software R. Proofs of the theorems of sections \ref{sec_mm} and \ref{sec_dd} are given in Appendix \ref{sec_append}.

\paragraph*{Radon transform}Here, we introduce the Radon transform and collect some of its properties. A more detailed introduction can be found in \cite{Natterer1986}. \\
Let us denote the unit sphere of $\R^d$ by $\Sphere^{d-1}:=\{y\in \R^d| \|y\| =1\}$, where $d \in \N$, $\|\cdot\|$ is the Euclidean norm and $\scalarpr{\cdot}{\cdot}$ the Euclidean scalar product.
Further, let  $\Lp{1}(\R^d)$ be the set of all measurable, integrable, complex-valued functions on $\R^d$ endowed with the norm $\Lpn{h}{1}{\R^d}:=\int_{\R^d} |h(y)| \inte\lambda^d(y)$ for $h\in \Lp{1}(\R^d).$ \\
From now on, let $d\in \N$, $d\geq 2$; then we define the Radon transform of $h\in \Lp{1}(\R^d)$ as the function $\radon{h}:\Sphere^{d-1}\times \R \rightarrow \C$ defined by 
\begin{align}\label{eq:radon}
    \radon{h}(s, u):= \int_{\{v\in \R^d| \scalarpr{v}{s}=u\}} f(v) \inte v,\quad \text{for }(s,u) \in \Sphere^{d-1}\times \R,
\end{align}
where the integration is taken over the hyperplane $\{v\in \R^d| \scalarpr{v}{s}=u\}$ with respect to the hypersurface measure. For example, for $d=2$ and $s^{\perp}\in \Sphere^{1}$, such that $\langle s,s^{\perp}\rangle=0$, we have
$$\radon{h}(s, u)= \int_{\R} h\left(us+vs^{\perp}\right) \inte\lambda(v).$$ This parametrisation can be generalised for all $d\geq 2.$
Further, we set $\radonind{s}{h}(u):=\radon{h}(s,u)$  for any fixed point $s \in \Sphere^{d-1}$.
\begin{ex}[Radon transform of a multivariate normal distribution]\label{ex:radon} \phantom{Hello} \\
   Let $\mu\in \R^d$, $\sigma\in \R_+:=(0,\infty)$ and $h(y):=(2\pi\sigma^2)^{-d/2}\exp(-\|y-\mu\|^2/(2\sigma^2)), y\in \R^d$, the density of a multivariate normal distribution. Then, for $(s,u)\in \mathcal Z:=\Sphere^{d-1}\times \R$, we get
   $$\radon{h}(s,u) = \frac{1}{\sqrt{2\pi\sigma^2}} \exp\left(-\frac{|u-\scalarpr{\mu}{s}|^2}{2\sigma^2}\right).$$
   In other words, for $s\in\Sphere^{d-1}$ fixed, $\radonind{s}{h}$ is the density of the normal distribution with mean $\scalarpr{\mu}{s}$ and variance $\sigma^2$ ($\mathrm{N}_{(\scalarpr{\mu}{s}, \sigma^2)}$ for short).
\end{ex}
Next, we study the relationship between the Fourier and the Radon transform. Indeed, let us denote by $\Lp{2}(\R^d)$ the set of all measurable, square-integrable, complex-valued functions on $\R^d$ endowed with the  norm $\Lpn{h}{2}{\R^d}^2:= \int_{\R^d} |h(y)|^2 \inte\lambda^d(y)$ and the corresponding inner product $\scalarpr{\cdot}{\cdot}_{\Lp{2}(\R^d)}$. As usual, we define the Fourier transform as the isomorphism $\mathcal F_d :\Lp{2}(\R^d)\rightarrow \Lp{2}(\R^d)$ given as the extension of $$\Fourier{d}{h}(t):= \int_{\R^d} \exp(i\scalarpr{t}{y}) h(y)\mathrm{d}\lambda^d(y), \quad \text{for } h\in \Lp{1}(\R^d)\cap\Lp{2}(\R^d), t\in \R^d.$$
Then, the so-called projection theorem holds true that is for $h\in \Lp{1}(\R^d)\cap\Lp{2}(\R^d)$, $(s,u)\in \mathcal Z$ that
\begin{align} \label{eq:proj:thm}\Fourier{1}{\radonind{s}{h}}(u) = \Fourier{d}{h}(us).
\end{align}
A detailed proof of \eqref{eq:proj:thm} can be found in \cite{Natterer1986}.
Further, the inverse of the Fourier transform $\mathcal F_d^{-1}: \Lp{2}(\R^d)\rightarrow \Lp{2}(\R^d)$ is explicitly given for $H\in \Lp{1}(\R^d)\cap \Lp{2}(\R^d)$ through $\mathcal F_d^{-1}[H](y)= (2\pi)^{-d} \int_{\R^d} \exp(-i\scalarpr{y}{t}) H(t) \mathrm{d}\lambda^d(t)$ and $(h_1*h_2)(y)=(2\pi)^{-d} \int_{\R^d} \exp(-i\scalarpr{y}{t}) \Fourier{d}{h_1}(t) \Fourier{d}{h_2}(t) \mathrm{d}\lambda^d(t)$ for $h_1,h_2 \in \Lp{1}(\R^d)\cap \Lp{2}(\R^d)$, where $(h_1*h_2)(y):= \int_{\R^d} h_1(y-z)h_2(z)\mathrm{d}\lambda^d(z)$ denotes the usual convolution of two $\Lp{1}(\R^d)$-functions. Further, we will frequently use the following equation
\begin{align}\label{eq:pol:cor}
\int_{\R^d} f(x) \mathrm{d} \lambda^d(x) = \int_0^{\infty} r^{d-1} \int_{\Sphere^{d-1}} f(r\omega) \surfacemeasure{d-1}{\omega}\mathrm{d}r,
\end{align}
where $\sigma^{d-1}$ denotes the surface measure on $\Sphere^{d-1}$. For a detailed discussion of this topic, see \cite[Chapter 7.2]{Natterer1986}. Next, we introduce the estimator presented in \cite{cavalier2000efficient} and discuss possible extensions.
\paragraph*{Estimation strategy} In this work, we study the estimator, proposed by \cite{cavalier2000efficient}, given by 
\begin{align}\label{eq:est}
			\widehat{f}_m(x) &:= \frac{1}{n} \sum_{i=1}^{n} K_m\left(\scalarpr{S_i}{x} - U_i\right), \quad
			K_m := \invFourier{1}{\indikator_{\ball{m}}\frac{1}{(2 \pi)^{d-1}} \frac{\rho_d}{2} \vert \cdot \vert^{d-1}},
\end{align} 
where $m \in \R_+$, $x\in \R^d$ and $\ball{m}:= \{x \in\R^d \vert \|x\| \leq m\}$. Here, $K_m(u)= \rho_d(2\pi)^{-d}\int_0^{m}r^{d-1}\cos(ur)\mathrm{d}r, u\in \R$ and $\rho_d = 2 \pi^{d/2}/\Gamma\left(d/2\right)$. By construction, we have $\EWind{\radon{f}}{\widehat f_m(x)}= \invFourier{d}{\indikator_{\ball{m}}\Fourier{d}{f}}(x)=:f_m(x)$, see \cite{cavalier2000efficient}.
\begin{rem}
The estimator $\widehat f_m(x)$ is a typical kernel-type density estimator in a statistical inverse problem. More precisely, \cite{BelomestnyGoldenshluger2020}, \cite{Brenner-MiguelComteJohannes2023} and \cite{ButuceaComte2009} studied similar constructions. In general, for a family $(\mathcal K_m)_{m\in \R_+}$ such that $(\mathcal K_m|\cdot|^{d-1})_{m\in \R_+}\subseteq \Lp{1}(\R) \cap \Lp{2}(\R)$ and $\lim_{m\rightarrow \infty} \mathcal K_m(t)=1,t\in \R$, we can consider the kernel estimator
$$ \widetilde f_m(x):= \frac{1}{n}\sum_{i=1}^n \widetilde K_m(\scalarpr{S_i}{x}-U_i),\quad \widetilde K_m:= \invFourier{1}{\mathcal K_m\frac{1}{(2 \pi)^{d-1}} \frac{\rho_d}{2} \vert \cdot \vert^{d-1}}$$
with $\mathcal K_m:=\indikator_{[-m,m]}$ being one example. For the sake of simplicity, we stay with this particular choice of $\mathcal K_m$.
\end{rem}
The following result on the upper bound of the mean squared error (MSE) of the estimator $\widehat f_m(x)$ can be found in a similar form in \cite{cavalier2000efficient}. We present the result in a different form and in Section \ref{a:intro} an alternative to the proof structure found in \cite{cavalier2000efficient}.
\begin{prop}\label{prop:mse}
	Let $\Fourier{d}{f} \in \Lp{1}(\R^d)$. Then, for $m \in \R_+$ and $x\in \R^d$ we have
	\begin{align}\label{ineq:mse}
		\EWindlr{\radon{f}}{\left\vert \widehat{f}_m(x) - f(x) \right\vert^2} 
		\leq& \frac{1}{(2\pi)^{2d}}\Lpn{\indikator_{{\ballc{m}}}\Fourier{d}{f}}{1}{\R^d}^2 + \mu_m\frac{m^{2d-1}}{n}, 
	\end{align}
	where $\mu_m:=\EWind{\radon{f}}{mK_1^2(m(\scalarpr{S_1}{x}-U_1))}\leq \rho_d((2\pi)^{2d}(2d-1))^{-1}(\rho_d+\|\Fourier{d}{f}\|_{\Lp{1}(\R^d)})$. 
\end{prop}
Let us briefly comment on the last result. The first summand of the right-hand side of the inequality \eqref{ineq:mse} is a typical bound for the squared bias term $|f(x)-f_m(x)|^2$ in the context of local kernel deconvolution estimators, compare \cite{Meister2009}. The second summand, the bound of the variance $\Varind{\radon{f}}{\widehat f_m(x)}$, is characterized by the term $m^{2d-1}n^{-1}$ as the expectations $\mu_m$ are bounded over $m\in \R_+$. More precisely, for any sequence $(m_n)_{n\in \N}$ with $m_n^{2d-1}n^{-1}\rightarrow 0$ and $m_n\rightarrow \infty$ for $n$ going to infinity, we deduce that $\EWind{\radon{f}}{\vert \widehat{f}_{m_n}(x) - f(x) \vert^2}\rightarrow 0$. In other words, $\widehat f_{m_n}(x)$ is a consistent estimator of $f(x)$. Nevertheless, without any further assumption of the density $f$, the rate of convergence of the MSE can be arbitrary slow. In the following section, we consider a regularity assumption expressed through Sobolev spaces to address this issue.
\section{Minimax theory}\label{sec_mm}
In this section, we introduce and study Sobolev ellipsoids and show the minimax optimality of our estimator over these function classes. Similar results for other statistical inverse problems can be found in \cite{ButuceaComte2009} for an additive noise model and in \cite{Brenner-MiguelComteJohannes2023} for the model of multiplicative measurement errors.\\
In contrast to the Sobolev class, a class of exponentially decaying Fourier transforms has been considered in \cite{cavalier2000efficient} leading to almost parametric rates. Similar to \cite{cavalier2000efficient}, we start by stating a rate of convergence over the Sobolev ellipsoids and show its minimax optimality by stating a lower bound of the minimax risk.\\
Indeed, let us consider for $L$, $\beta \in \R_+$ the Sobolev ellipsoids $\Sobolev{\beta}{L}$ defined by
\begin{align}\label{eq:sobo:el}
	\Sobolev{\beta}{L} &:= \sets{ f \in \Lp{2}\left(\R^d\right) }{\int_{\R^d} \left( 1 + \| t \|^2 \right)^{\beta} \left\vert \Fourier{d}{f}(t)\right\vert^2 \Lebesgue{d}{t} \leq L}.
\end{align}
If $f\in \Sobolev{\beta}{L}$, with $\beta>d/2$, we have $\Fourier{d}{f} \in \Lp{1}(\R^d)$ and, by application of the Cauchy-Schwarz inequality, we get
$$\Lpn{\indikator_{\ballc{m}} \Fourier{d}{f}}{1}{\R^d}^2 = \left|\int_{\ballc{m}} \left|\Fourier{d}{f}(t)\right|\frac{(1+\|t\|^2)^{\beta/2}}{(1+\|t\|^2)^{\beta/2}}\mathrm{d}\lambda^d(t)\right|^2 \leq C(L,\beta,d) m^{-2\beta+d},$$
where $C(L,\beta,d)>0$ is a positive constant depending on $L$, $\beta$ and $d$. 
Now, we can show the following Corollary, which is a direct consequence of Proposition \ref{prop:mse}. Its proof is thus omitted.
\begin{cor}[Upper bound on the minimax risk]\label{cor:upp:bou}
    Let $\beta>d/2$ and $L\in\R_+$. Then for $m_o:= n^{1/(2\beta+d-1)} $, we get
    $$\sup_{f\in \Sobolev{\beta}{L}} \EWindlr{\radon{f}}{\left\vert \widehat{f}_{m_o}(x) - f(x) \right\vert^2}  \leq C(L, \beta, d) n^{-\frac{2\beta-d}{2\beta+d-1}}.$$
\end{cor}
In Corollary \ref{cor:upp:bou}, we have seen that the estimator $\widehat f_{m_o}(x)$ achieves a rate of $n^{-(2\beta-d)/(2\beta+d-1)}$ uniformly over the ellipsoid $\Sobolev{\beta}{L}$. To show that this rate cannot be improved by any estimator based on the sample $\left(\left(S_i, U_i\right)\right)_{i\in \llbracket n\rrbracket}$, where $\llbracket n \rrbracket := [1,n] \cap \N $, we show the following lower bound.
\begin{thm}[Lower bound on the minimax risk]\label{thm:low:bou}
	Let $x\in\R^d$ and let $L$, $\beta \in \R_+$. Then, there exist constants $L_{\beta, d}>0$ and $c\left(L, \beta, d\right) > 0$ depending on $\beta$, $d$, respectively  $L$, $\beta$, $d$, such that for all $L \geq L_{\beta, d}$ and for any estimator $\widehat{f}\left(x\right)$ of $f\left(x\right)$ based on an i.i.d. sample $\left(\left(S_i, U_i\right)\right)_{i\in \llbracket n\rrbracket}$ it holds that
	\begin{align*}
		\sup_{f \in \Sobolev{\beta}{L}} \EWindlr{\radon{f}}{\left|\widehat{f}\left(x\right) - f\left(x\right)\right|^2} \geq c(L, \beta, d) n^{-\frac{2\beta-d}{2\beta+d-1}}.
	\end{align*}
\end{thm}
Theorem \ref{thm:low:bou} implies that the rate of $n^{-(2\beta-d)/(2\beta+d-1)}$ is due to the ill-posedness of the underlying inverse problem rather a property of our proposed estimator. Moreover, we have shown that the estimator $\widehat f_{m_o}(x)$ is minimax-optimal with respect to the class $\Sobolev{\beta}{L}$. Nevertheless, the choice $m_o=n^{1/(2\beta+d-1)}$ in Corollary \ref{cor:upp:bou} is still dependent on the regularity of the unknown density $f$ which is again unknown. \\
In the next section, we therefore aim to present a fully data-driven choice of the parameter $m\in \R_+$ based only on the observations $((S_i,U_i))_{i\in \llbracket n \rrbracket}$.
\section{Data-driven method}\label{sec_dd}
In this section, we state a data-driven procedure for the choice of the parameter $m\in \R_+$ inspired by the work  \cite{GoldenshlugerLepski2011}. More precisely, for $x\in \R^d$, we  propose a data-driven choice of $m\in \R_+$ based only on the observations $((S_i, U_i))_{i\in \llbracket n \rrbracket}$. \\
In fact, let us reduce the set of eligible dimension parameters to
\begin{align}\label{eq:Mn}
	\Mn &:= \sets{m \in \N}{m \leq n^{\frac{1}{2d-1}}} \text{ and set } \mathrm{M}_n = \max \Mn.
\end{align}
We define for a numerical constant $\chi>0$ and $m\in \Mn$ the function $V(m) := \chi (1+\mu_m) m^{2d-1} \log(n)/n$ which mimics the variance term up to a logarithmic factor. Furthermore, we set
\begin{align*}
	A(m) := \hspace*{-0.2cm}\sup_{m^{\prime} \in \Mn} \left(\left|\widehat{f}_{m^{\prime} \wedge m}(x) - \widehat{f}_{m^{\prime}}(x)\right|^2 - V\left(m^{\prime}\right)\right)_+.
 \end{align*}
As the quantity $(\mu_m)_{m\in \R_+}$ is unknown, we replace it with its empirical counterpart $\widehat\mu_m:=n^{-1} \sum_{j\in \llbracket n\rrbracket} mK_1^2(m(\scalarpr{S_j}{x}-U_j))$ and define $\widehat V(m) := 2\chi (1+\widehat\mu_m)m^{2d-1} \log(n)/n$ and
\begin{align}\label{def:hVhA}
	\widehat A(m) := \hspace*{-0.2cm}\sup_{m^{\prime} \in \Mn} \left(\left|\widehat{f}_{m^{\prime} \wedge m}(x) - \widehat{f}_{m^{\prime}}(x)\right|^2 - \widehat V\left(m^{\prime}\right)\right)_+.
 \end{align}
For abbreviation, we write $A(m)$, $\widehat{A}(m)$, $V(m)$, $\widehat{V}(m)$ and $\hat{\mu}_m$ instead of $A(m,x)$, $\widehat{A}(m,x)$, $V(m,x)$, $\widehat{V}(m, x)$ and $\hat{\mu}_m(x)$, respectively, although it depends on $x\in \R^d$.
Below we study the estimator $\widehat f_{\widehat m}(x)$ with the cut-off parameter selected by
\begin{align}\label{eq:adap:m}
	\widehat{m} &:= \argmin_{m \in \Mn} \left\{\widehat A(m) + \widehat V(m)\right\},
\end{align}
which depends only on observable quantities and hence $
\widehat f_{\widehat m}(x)$  is completely data-driven. Note that we only write $\widehat{m}$ instead of $\widehat{m}(x)$ to simplify notation.

\begin{thm}[Data-driven choice of $m\in \R_+$] \label{theo:adap} Let $\Fourier{d}{f}\in \Lp{1}(\R^d)$. Then for all $\chi\geq 24$
	\begin{align*}
	\EWindlr{\radon{f}}{\left|\widehat f_{\widehat m}(x)-f(x)\right|^2} \leq 16\inf_{m\in \Mn}\left(\frac{\|\mathds{1}_{\ballc{m}}\Fourier{d}{f}\|^2_{\Lp{1}(\R^d)}}{(2\pi)^{2d}} +V(m)\right) + \frac{C(f,\chi, d)}{n},
	\end{align*}
 where $C(f,\chi, d)>0$ is a positive constant depending on $\Lpn{\Fourier{d}{f}}{1}{\R^d},\chi$ and $d$.
\end{thm}
It is worth emphasising that the class $\Mn$ is rich enough in the sense  that the data-driven estimator $\widehat f_{\widehat m}(x)$ achieves the minimax rate up to a logarithmic factor. Indeed, for $\beta> d/2$ we have $m_o= (n/\log(n))^{1/(2\beta+d-1)} \leq n^{1/(2d-1)}$ which implies with Theorem \ref{theo:adap} the following Corollary.
\begin{cor}\label{cor:upp:bou:adap}
    Let $x\in \R^d$, $L\in\R_+$ and $\beta>d/2$. Then, there exists a constant $C(L,\beta, d)>0$ depending on $L$, $\beta$ and $d$ such that
    $$\sup_{f\in \Sobolev{\beta}{L}} \EWindlr{\radon{f}}{\left|\widehat f_{\widehat m}(x)-f(x)\right|^2} \leq C(L, \beta, d) \left(\frac{n}{\log(n)}\right)^{-\frac{2\beta-d}{2\beta+d-1}}.$$
\end{cor}
\begin{rem}
    Even if we assume that the density $f$ belongs to $\Sobolev{\beta}{L}$ with $\beta> d/2$ and $L\in\R_+$, 
    the convergence rate in the adaptive procedure deteriorates by a ($\log(n)$)-factor compared to the optimal pointwise rate under the same assumption, cf. Corollary \ref{cor:upp:bou} and Theorem \ref{thm:low:bou}. However, this is a typical phenomenon in the literature, see \cite{BUTUCEA2000} and \cite{butucea2001}.
\end{rem}

To show that the logarithmic loss in the convergence rate in Corollary \ref{cor:upp:bou:adap} is unavoidable, we state a lower bound as well. 

\begin{thm}[Lower bound for data-driven estimators]\label{thm:low:adap}
    Let $x \in\R^d$ and $\beta>d/2$. Then, there exists a constant $L_{\beta, d}> 0$ depending on $\beta$, $d$ such that for all $L\geq L_{\beta, d}$ holds:
    If a sequence of estimators $\{\widehat f_n(x)\}_{{n\in \N}}$ of $f(x)$ based on the data $((S_i,U_i))_{i \in \llbracket n \rrbracket}$  satisfies
    $$\sup_{n\in \N} \sup_{f \in \Sobolev{\beta}{L}} \EWindlr{\radon{f}}{\left\vert \widehat f_n(x) - f(x) \right\vert^2} n^{\frac{2\beta-d}{2\beta+d-1}}\leq \mathfrak C $$ for a positive constant $\mathfrak C>0$, then for any $\beta' \in (d/2, \beta) $ there exists a constant $\mathfrak c=\mathfrak c(\mathfrak C, \beta, \beta', d)>0$ such that
    $$\liminf_{n\rightarrow \infty}\sup_{f \in \Sobolev{\beta'}{L}} \EWindlr{\radon{f}}{\left|\widehat f_{n}(x) - f(x)\right|^2} \left(\frac{n}{\log(n)}\right)^{\frac{2\beta'-d}{2\beta' + d - 1}}\geq \mathfrak c. $$
\end{thm}
\section{Numerical simulation}\label{sec_nu}
In this section, we demonstrate the reasonable behaviour of the estimator $\widehat f_{\widehat m}(x)$, defined in \eqref{eq:est} and \eqref{eq:adap:m} for the case $d=2$. To do so, we exploit two examples of densities $f$ in a Monte Carlo simulation with 500 iterations and varying values of $x \in \R^2$ and sample sizes $n\in \N$. For the densities we consider
\begin{enumerate}
\item \textit{Bivariate normal distribution.} $f(x)=(2\pi)^{-1} \exp(-\| x\|^2/2)$, for $ x\in \R^2$, and
\item \textit{Uniform distribution.} $f( x)=\pi^{-1} \mathds 1_{\Sphere^1}(x)$, for $ x \in \R^2$.
\end{enumerate}
The first example fulfills all assumptions we have made on the density $f$. The second example, that of a uniform distribution on the unit sphere, does not fulfil the assumption that $\mathcal F_d[f]\in \mathds L^1(\R^2)$, as it is not continuous on the boundary of $\Sphere^1$. Nevertheless, we include this example to visualize that the estimator behaves reasonably even if not all assumptions are satisfied.
Based on a preliminary simulation study, we choose $\chi=0.003$ in all examples. Moreover, we use $\mathcal M_n=\{0.1 m| m\in \N, m \leq n\}$ instead of the definition in \eqref{eq:Mn}. Regarding the proof of Theorem \ref{theo:adap}, this only affects the positive constant $C(f, \chi, d)$ and has no effect on the rate.
\begin{minipage}[t]{\textwidth}
	\centering{\begin{minipage}[t]{0.32\textwidth}
			\includegraphics[width=\textwidth,height=40mm]{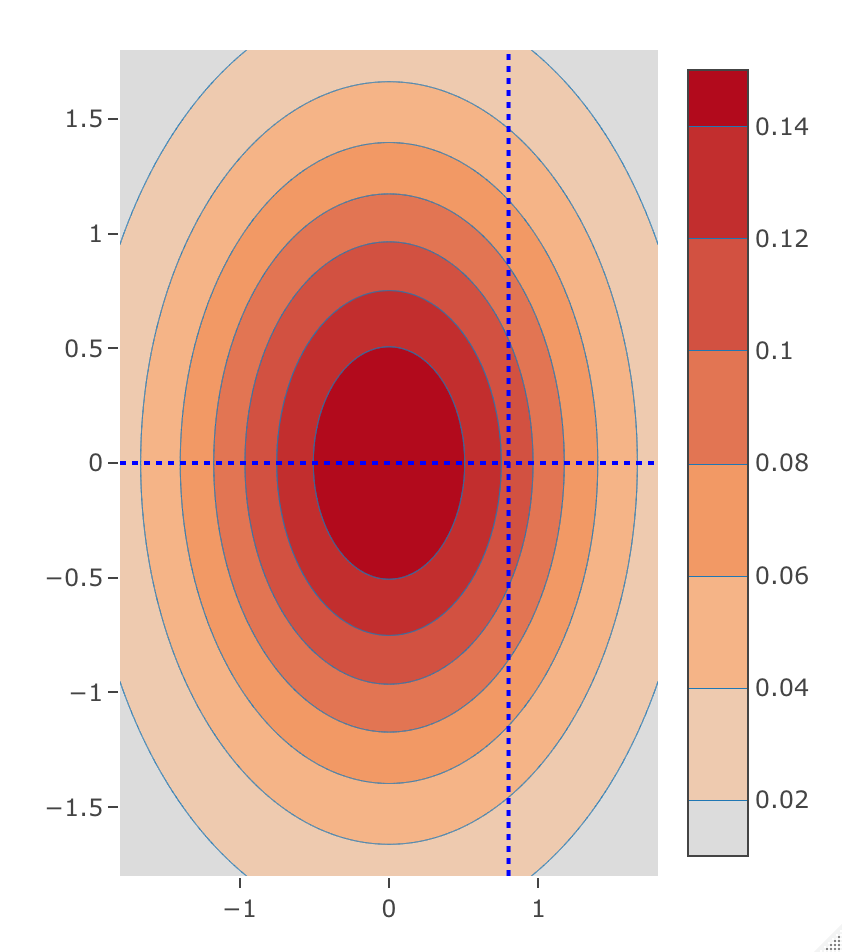}
		\end{minipage}
		\begin{minipage}[t]{0.32\textwidth}
			\includegraphics[width=\textwidth,height=42mm]{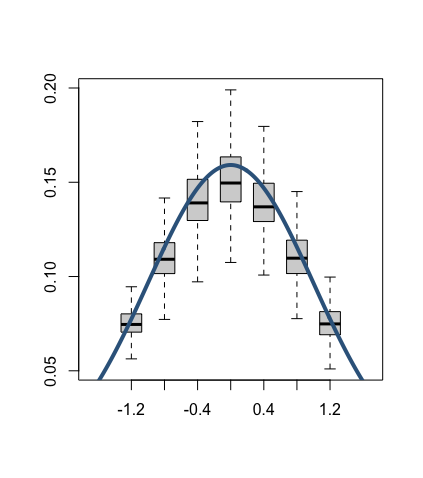}
	\end{minipage}
\begin{minipage}[t]{0.32\textwidth}
	\includegraphics[width=\textwidth,height=42mm]{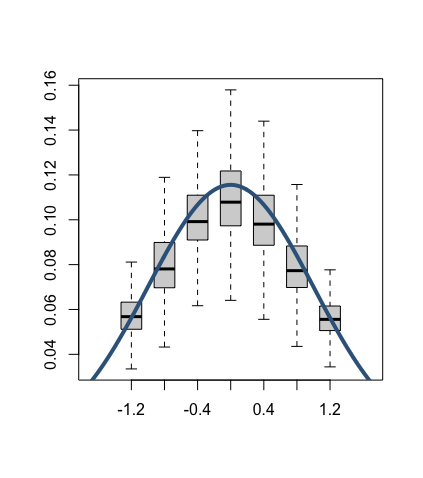}
\end{minipage}}
	\centering{\begin{minipage}[t]{0.32\textwidth}
			\includegraphics[width=\textwidth,height=40mm]{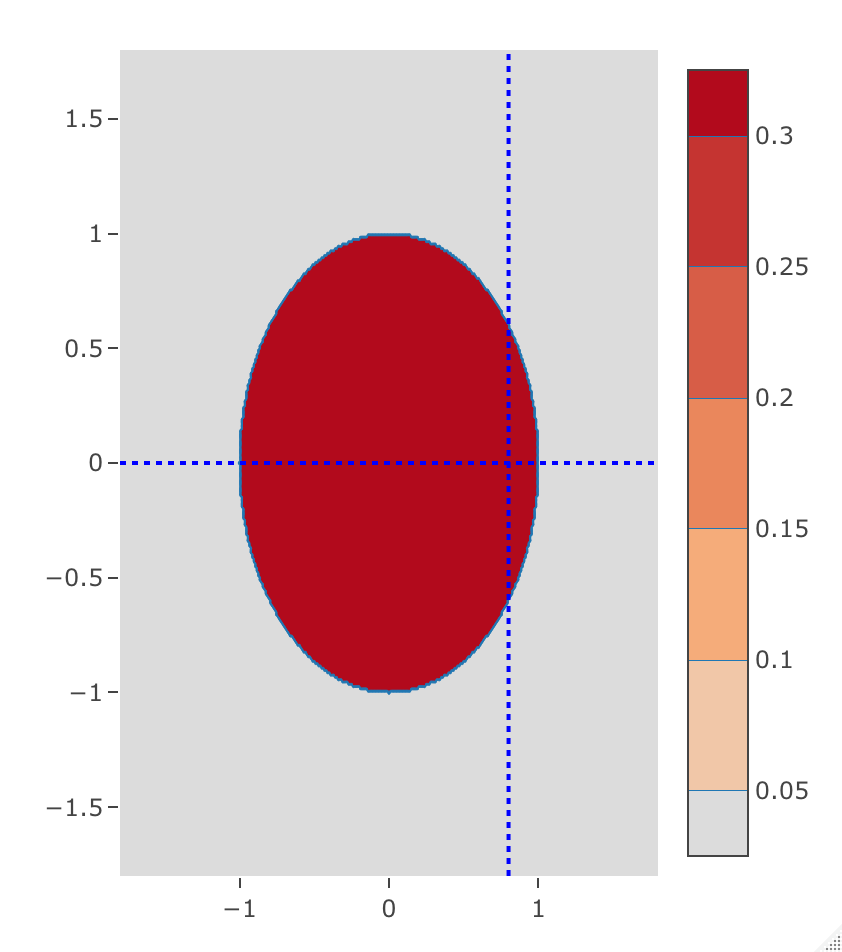}
		\end{minipage}
		\begin{minipage}[t]{0.32\textwidth}
			\includegraphics[width=\textwidth,height=42mm]{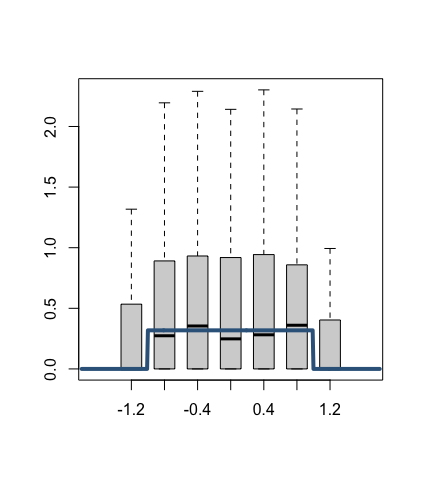}
	\end{minipage}
\begin{minipage}[t]{0.32\textwidth}
	\includegraphics[width=\textwidth,height=42mm]{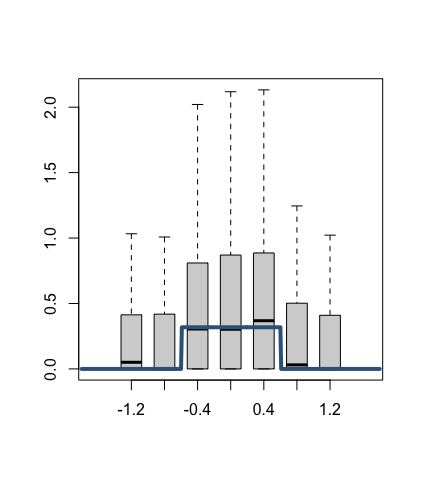}
\end{minipage}}
	\captionof{figure}{\label{figure:1}Considering the examples $(i)$ (top) and $(ii)$ (bottom). For each example, we illustrate the true density (left), the horizontal cut (middle) and the vertical cut (right). For the middle and right column, a Monte Carlo simulation with $n=10.000$ and 500 iterations of $\widehat f_{\widehat m}(x)$ has been performed. The dark blue curve is the true function $f$.}
\end{minipage}
In Figure \ref{figure:1}, we see a convenient behaviour of our fully data-driven estimator $\widehat f_{\widehat m}(x)$ for varying $x\in \R^2$. Comparing (i) and (ii) the estimation of the value $f(x)$ seems to be more complex in setting (i). Although the example (ii) does not fulfil the assumptions, it can be seen that the support of the density is correctly identified through the proposed fully data-driven estimator.

\begin{minipage}[t]{\textwidth}
	\centering{\begin{minipage}[t]{0.49\textwidth}
			\includegraphics[width=\textwidth,height=72mm]{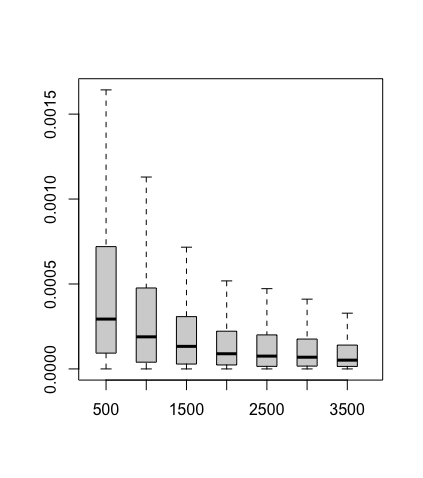}
		\end{minipage}
		\begin{minipage}[t]{0.49\textwidth}
			\includegraphics[width=\textwidth,height=72mm]{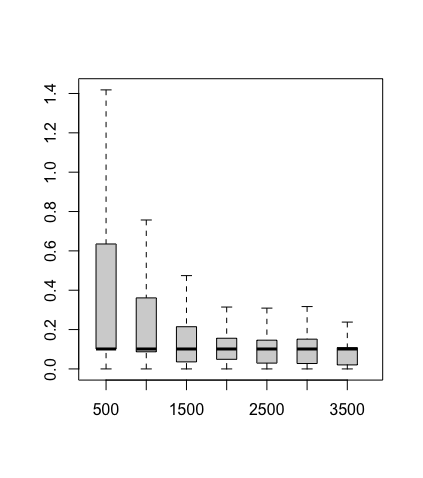}
	\end{minipage}}
	\captionof{figure}{\label{figure:2}Considering the examples $(i)$ (left) and $(ii)$ (right). For each example, we illustrate $|f((0,0))-\widehat f_{\widehat m}((0,0))|^2$ using a  Monte Carlo simulation with 500 iterations and varying sample size $n\in \{500,1.500,2.000,2.500,3.000,3.500\}$.}
\end{minipage}
    Figure \ref{figure:2} illustrates the asymptotic behaviour of the risk using a Monte Carlo simulation. From these illustrations, it seems that the risk is decreasing in the sample size $n\in \N$ which suggests that our estimation strategy is consistent, even if the considered assumptions are not fulfilled.
    
\paragraph*{Acknowledgement}
    This work is supported by the Deutsche Forschungsgemeinschaft (DFG, German Research Foundation) under Germany’s Excellence Strategy EXC 2181/1 - 390900948 (the Heidelberg STRUCTURES Excellence Cluster) and by the Research Training Group ”Statistical Modeling of Complex Systems”.
    This research has been partially funded by Deutsche Forschungsgemeinschaft (DFG) - Project-ID 410208580 - IRTG2544 ("Stochastic Analysis in Interaction").
\newpage
\section{Appendix}
\label{sec_append}

\subsection{Useful inequalities}
For the proof of Lemma \ref{lem:concen} below, we need the
   following Bernstein inequality, which in this form can be found in
   \cite{Comte2017} and is based on a similar formulation in
   \cite{BirgeMassart1998}.
\begin{lem}[Bernstein inequality] \label{ineq: bernstein}\leavevmode\newline 
	Let $T_1,\ldots,T_n$ be independent random variables, and we define $S_n(T) := \sum_{i=1}^n \left(T_i - \EW{T_i}\right)$. Then, for any $\eta > 0$, we get
	\begin{align*}
		\PP{\left\vert S_n(T) - \EW{S_n(T)} \right\vert \geq n \eta} 
		& \leq 2 \max\left\{\exp\left(-\frac{n \eta^2}{4v^2}\right), \exp\left(-\frac{n \eta}{4b}\right)\right\},
	\end{align*}
	where $\frac{1}{n}\sum_{i=1}^n \EW{\left\vert T_i \right\vert^m} \leq \frac{m!}{2}v^2b^{m-2}$, for all $m \geq 2$, for some positive constants $v$ and $b$. \newline
	If the random variables $T_i$,$i \in \llbracket n \rrbracket$, are identically distributed, the previous condition can be replaced by $\Vari{T_1} \leq v^2$ and $\vert T_1 \vert \leq b$. 
\end{lem}

\subsection{Proof of Section \ref{sec_intro}}\label{a:intro}
\begin{proof}[Proof of Proposition \ref{prop:mse}]
First, the usual bias-variance decomposition leads to 
	\begin{align*}
		\EWindlr{\radon{f}}{\left\vert \widehat{f}_m(x) - f(x)\right\vert^2} \nonumber
			= \left|f_m(x) - f(x)\right|^2 + \EWindlr{\radon{f}}{\left| f_m(x) - \widehat{f}_m(x)\right|^2}
	\end{align*}
	For the first term, the squared bias term, we obtain 
 \begin{align*}
     |f_m(x)-f(x)| \leq \frac{1}{(2\pi)^{d}} \int_{\ballc{m}} |\Fourier{d}{f}(t)| \inte\lambda^d(t)=  \frac{1}{(2 \pi)^{d}}\Lpn{\indikator_{\ballc{m}} \Fourier{d}{f}}{1}{\R^d}.
 \end{align*}
By construction of $\widehat f_m$ and as $((S_i, U_i))_{i\in \llbracket n\rrbracket}$ are i.i.d. the variance term is bounded, we get
\begin{align*}
    \EWindlr{\radon{f}}{\left| f_m(x) - \widehat{f}_m(x)\right|^2} &\leq \frac{1}{n}  \EWindlr{\radon{f}}{K_m^2\left(\scalarpr{S_1}{x} - U_1\right)} \\
    &= \frac{1}{n\rho_d} \int_{\Sphere^{d-1}}\int_{\R} K_m^2\left(\scalarpr{s}{x} - u\right) \radon{f}(s,u) \Lebesgue{}{u} \surfacemeasure{d-1}{s}\\
    &=\frac{m^{2d-1}}{n} \mu_m
\end{align*}
as $\mu_m=\rho_d^{-1} \int_{\Sphere^{d-1}}\int_{\R} mK_1^2\left(m(\scalarpr{s}{x} - u\right)) \radon{f}(s,u) \hspace{-4pt}\Lebesgue{}{u} \hspace{-4pt} \surfacemeasure{d-1}{s}$ and $K_m(u)=m^d K_1(um), u\in \R$.  It remains to show that $\mu_m$ is bounded for all $m\in \R_+$. Indeed, we conclude for $s\in \Sphere^{d-1}$
\begin{align*}
    m\int_{\R} K_1^2(m(\scalarpr{s}{x}-u))  \radonind{s}{f}(u) \Lebesgue{}{u} 
    = \frac{1}{2\pi}\int_{\R} e^{-it\scalarpr{s}{x}} \Fourier{1}{K_1^2}(t/m) \Fourier{d}{f}(st)\Lebesgue{}{t}.
\end{align*}
Thus
\begin{align*}
    &m\left|\int_{\Sphere^{d-1}}\int_{\R} K_1^2\left(m(\scalarpr{s}{x} - u)\right) \radonind{s}{f}(u) \Lebesgue{}{u} \surfacemeasure{d-1}{s}\right| \\
    &\leq \frac{1}{2\pi}\int_{\Sphere^{d-1}}\int_{\R}|\Fourier{1}{K_1^2}(t/m) \Fourier{d}{f}(st)|\Lebesgue{}{t}\surfacemeasure{d-1}{s}\\
    &\leq \frac{\|K_1\|_{\Lp{2}(\R)}^2}{2\pi}\int_{\Sphere^{d-1}}\int_{\R}|\Fourier{d}{f}(st)|\Lebesgue{}{t}\surfacemeasure{d-1}{s}.
\end{align*}
The statement follows applying \eqref{eq:pol:cor} and using $\|K_1\|_{\Lp{2}(\R)}^2=\rho_d^2/((2\pi)^{2d-1}2(2d-1))$
\begin{align}\label{ineq:L2:radon}
    &\phantom{\leq}\int_{\Sphere^{d-1}}\int_{\R}|\Fourier{d}{f}(st)|\Lebesgue{}{t}\surfacemeasure{d-1}{s} \notag\\
    &\leq 2\int_{\Sphere^{d-1}}\int_{(0,\infty)}|\Fourier{d}{f}(st)|\left(1\vee t^{d-1}\right)\Lebesgue{}{t}\surfacemeasure{d-1}{s} \notag\\
    &\leq 2\left(\rho_d+\|\Fourier{d}{f}\|_{\Lp{1}\left(\R^d\right)}\right),
\end{align}
where $a\vee b := \max\{a, b\}$.
\end{proof}
\subsection{Proof of Section \ref{sec_mm}}\label{a:mm}

\begin{proof}[Proof of Theorem \ref{thm:low:bou}]
	First, we outline the main steps of the proof, which are captured in multiple lemmata and proven after this proof is completed. We will construct a function in the Sobolev ellipsoids $\Sobolev{\beta}{L}$ by a perturbation of the density $f_0:\R^d \rightarrow \R$ with a small bump, such that the distance $\left\vert f_1\left(x\right) - f_0\left(x\right)\right\vert$ for any $x\in\R^d$ and the Kullback--Leibler divergence of their induced distributions can be bounded from below and above, respectively. The claim then follows by applying Theorem 2.2 in \cite{Tsybakov2008}. We use the following construction, which we present first. \\
	For $f_0$, we choose the density of the multivariate normal distribution, i.e.
	\begin{align}\label{eq:f:1}
		f_0(y) &= \frac{1}{\left(\sqrt{2 \pi}\right)^d} \exp\left(- \frac{1}{2} \scalarpr{x-y}{x-y}\right) \text{ for } x, y \in \R^d.
  \end{align}
	 Let $\psi \in C_0^\infty\left(\R\right)$ be a function with  $\mathrm{supp}(\psi)\subseteq[0,1]$, $\psi(0)=1$, $\|\psi\|_{\infty}=1$ and $\int_0^1 y^{d-1} \psi(y) \inte y = 0$. For a bump-amplitude $\delta > 0$ and for $h \in (0,1)$ (to be chosen later on) we define
    \begin{align}
		f_1(y) &:= f_0(y) + \delta h^{\beta - \frac{d}{2}} \Psi_{h, x}(y), \label{eq: f1 fourier radon}
	\end{align}
	where $x$, $y \in \R^d$ and $\Psi_{h, x}(y) := \psi(\| y - x \|/h)$. So far, we have  not presented a sufficient condition to ensure that our constructed function $f_1$ is in fact a density. This is covered by the following lemma.
	\begin{lem}\label{hilfslemma 1: f1 density radon pointwise}
		Let $0 < \delta \leq (2\pi e)^{-d/2}$. Then $f_1$, as in Equation \eqref{eq:f:1}, is a density.
	\end{lem}
	 Further, we can show that these densities lie inside the Sobolev ellipsoids $\Sobolev{\beta}{L}$ for $L \in \R_{+}$ big enough. This is captured in the following lemma.
	\begin{lem}\label{hilfslemma 2: f0 f1 in fourier sobolev radon pointwise}
		Let $s \in \R_+$. Then, there is a constant $L_{\beta, d} > 0$ depending on $\beta$, $d$ such that $f_0$ and $f_1$ belong to $\Sobolev{\beta}{L_{\beta, d}}$.
	\end{lem}
	The following lemma captures that the difference $\left\vert f_0\left(x\right) - f_1\left(x\right)\right\vert^2$ and the Kullback--Leibler divergence of their induced distributions is bounded from below and above, respectively. Let us denote by $\mathds{P}_{\radon{f}}$ the measure with density $\rho_d^{-1} \radon{f}$ and  by $f^{\ast}$ the density of the standard normal distribution.
	\begin{lem}\label{hilfslemma 3: f0 f1 bounds radon pointwise}
		Let $h \in (0,1)$. Then, for $x\in \R^d$,
		\begin{enumerate}
			\item[$(i)$] $\left\vert f_1\left(x\right) - f_0\left(x\right)\right\vert^2 = \delta^2 h^{2\beta-d}$ and
			\item[$(ii)$] $\KL{\mathds{P}_{\radon{f_1}}}{\mathds{P}_{\radon{f_0}}} \leq \delta^2 h^{2\beta+d-1}(\rho_d+(2\pi)^d)(2\pi e)^{-1/2}$.
		\end{enumerate}
	\end{lem}
	Selecting  $h := n^{-1/(2\beta+d-1)}$, it follows, using that $\left(\left(S_i,U_i\right)\right)_{i \in \llbracket n \rrbracket}$ are i.i.d., that
	\begin{align*}
		\KL{\mathds{P}_{\radon{f_1}}^n}{\mathds{P}_{\radon{f_0}}^n} &= n \KL{\mathds{P}_{\radon{f_1}}}{\mathds{P}_{\radon{f_0}}} \leq \widetilde{C}\left(L, \beta, d\right) \delta^2 \Lpinfty{\psi}^2 < \frac{1}{8},
	\end{align*}
 	for all $\delta<\delta(L,\beta,d)$ for some positive constants $\delta(L,\beta,d)$ depending on $L$, $\beta$, $d$. Therefore, we can use Theorem 2.2 of \cite{Tsybakov2008}, which yields
	\begin{align*}
		\inf_{\widehat{f}} \sup_{f \in \Sobolev{\beta}{L}} \PPindlr{\radon{f}}{\left\vert \widehat{f}\left(x\right) - f\left(x\right)\right\vert^2 > \frac{1}{2} \delta n^{-\frac{2\beta-d}{2\beta+d-1}}} &\geq \max\left\{\frac{1}{4} \exp\left(-\frac{1}{8}\right), \frac{1 - \sqrt{\frac{1}{16}}}{2} \right\} \\
		&= 0.375,
		\intertext{where the infimum is taken over all estimators $\widehat{f}$ of $f$, based on the i.i.d. data $\left(\left(S_i,U_i\right)\right)_{i\in\llbracket n \rrbracket}$. Using Markov's inequality gives}
		\inf_{\widehat{f}} \sup_{f \in \Sobolev{\beta}{L}} \EWindlr{\radon{f}}{\left\vert \widehat{f}\left(x\right) - f\left(x\right)\right\vert^2} &\geq c(L, \beta, d) n^{-\frac{2\beta-d}{2\beta+d-1}},
	\end{align*}
	where $c(L,\beta,d)$ is a positive constant depending on $L$, $\beta$ and $d$. This completes the proof. 
\end{proof}

\begin{proof}[Proof of Lemma \ref{hilfslemma 1: f1 density radon pointwise}]
	By an application of Equation \eqref{eq:pol:cor}, we deduce that 
	\begin{align*}
		\int_{\R^{d}} \Psi_{h, x}(y) \Lebesgue{d}{y} 
		&= h^{d} \int_{0}^\infty r^{d-1} \int_{\Sphere^{d-1}} \psi\left(\| r s \| \right) \surfacemeasure{d-1}{s} \inte r \\
		&=\rho_d h^{d} \int_0^\infty r^{d-1} \psi(r) \inte r=0
  \end{align*}
  implying $\int_{\R^d} f_1(y) \mathrm{d}\lambda^d(y)=\int_{\R^d} f_0(y)\mathrm{d}\lambda^d(y)=1$.
  Using that the function $\psi$ has support $\mathrm{supp}(\psi)$ in $[0,1]$ leads to $\mathrm{supp}\left(\Psi_{h,x}\right) \subseteq \left[x - h, x + h\right] := \bigtimes_{j=1}^d \left[x_{j} - h, x_{j} + h\right]$. For $y \not\in \left[x - h, x + h\right]$ holds  $f_1(y) = f_0(y) \geq 0$. Further, for $y \in \left[x - h, x + h\right]$ it holds that $\inf_{y \in \left[x - h, x + h\right]} f_0(y) \geq \inf_{y \in \left[x - 1, x + 1\right]} f_0(y) \geq (2\pi e)^{-d/2}$ and
    \begin{align*}
	f_1(y) &= f_0(y) + \delta h^{\beta-\frac{d}{2}} \Psi_{h, x}(y) \geq (2\pi e)^{-d/2} - \delta h^{\beta - \frac{d}{2}}. 
    \end{align*} 
    As  $h^{\beta-d/2} \leq 1$, choosing $\delta \in (0, (2\pi e)^{-d/2}]$ ensures the positivity of $f_1$. 
\end{proof}
\begin{proof}[Proof of Lemma \ref{hilfslemma 2: f0 f1 in fourier sobolev radon pointwise}]
	Our proof starts with the observation that for all $t \in \R^d$ we have $|\Fourier{d}{f_0}(t)| = \exp\left(-\frac{1}{2} \| t\|\right)$. Thus, for every $\beta \in \R_+$ there exists a positive constant $\widetilde{L}_{\beta,d}$ such that 
	\begin{align*}
		\int_{\R^d} \left|\Fourier{d}{f_0}(t)\right|^2 \left( 1+ \| t \|^2 \right)^{\beta} \Lebesgue{d}{t} &=: \widetilde{L}_{\beta,d},
	\end{align*}
	i.e. $f_0 \in \Sobolev{\beta}{L}$ for all $L\geq \widetilde{L}_{\beta,d}$. \newline
	Next, we consider $f_1$. From $|\Fourier{d}{\Psi_{h,x}}(t)|=|\Fourier{d}{\Psi_{1,0}}(ht)|, t\in \R^d$, we have
	\begin{align*}
		\Lpnlr{\Fourier{d}{f_0 - f_1}\left(1+\|\cdot\|^2\right)^{\beta/2}}{2}{\R^d}^2 &= 
			\delta^2h^{2\beta+d} \int_{\R^d} \left\vert \Fourier{d}{\Psi_{1,0}}\left(h t\right) \right\vert^2 \left(1 + \| t \|^2\right)^{\beta} \Lebesgue{d}{t} \\
			&= \delta^2 \int_{\R^d} \left\vert \Fourier{d}{\Psi_{1,0}}(t)\right\vert^2 \left(h^2 + \| t \|^2\right)^{\beta} \Lebesgue{d}{t}. 
   \end{align*}
    As $h^2\leq 1$, we conclude, by using the linearity of the Fourier transform,
    \begin{align*}
				\Lpnlr{\Fourier{d}{f_1}\left(1+\|\cdot\|^2\right)^{\beta / 2}}{2}{\R^d}^2 &\leq \left(\left(\widetilde{L}_{\beta,d}\right)^{1 / 2} + \delta  \Lpnlr{\Fourier{d}{\Psi_{1,0}}\left(1+\|\cdot\|^2\right)^{\beta / 2}}{2}{\R^d}\right)^2,
	\end{align*}
	i.e. for all $L \geq L_{\beta, d} := ((\widetilde{L}_{\beta,d})^{1 / 2} + \delta \Lpnlr{\Fourier{d}{\Psi_{1,0}}(1+\|\cdot\|^2)^{\beta / 2}}{2}{\R^d})^2$ holds $f_0$, $f_1 \in \Sobolev{\beta}{L}$. 
\end{proof}

\begin{proof}[Proof of Lemma \ref{hilfslemma 3: f0 f1 bounds radon pointwise}] 
\underline{For $(i)$}, we have
	\begin{align*}
		\left\vert f_0\left(x\right) - f_1\left(x\right)\right\vert^2 
		= \, \delta^2 h^{2\beta-d}\left\vert \psi(0) \right\vert^2 
		= \,\delta^2 h^{2\beta-d}. 
	\end{align*}
	\underline{Considering $(ii)$}, as $\KL{\mathds{P}_{\radon{f_1}}}{\mathds{P}_{\radon{f_0}}} \leq \chi^2\left(\mathds{P}_{\radon{f_1}}, \mathds{P}_{\radon{f_0}}\right)$ we have
	\begin{align*}
		\KL{\mathds{P}_{\radon{f_1}}}{\mathds{P}_{\radon{f_0}}} &\leq  \int_{\Sphere^{d-1}} \int_{\R}\frac{\left|\frac{1}{\rho_d} \radon{f_1-f_0}({s}, u)\right|^2}{\frac{1}{\rho_d} \radon{f_0}({s}, u)} \Lebesgue{}{u} \surfacemeasure{d-1}{{s}}\\
  &=\frac{\delta^2 h^{2\beta-d}}{\rho_d} \int_{\Sphere^{d-1}} \int_{\R} \frac{\left\vert\radon{\Psi_{h,x}}({s},u)\right\vert^2}{f^{\ast}(u-\scalarpr{x}{s})} \Lebesgue{}{u} \surfacemeasure{d-1}{{s}},
	\end{align*} 
	 where $f^{\ast}$ denotes the density of the standard normal distribution, i.e. $f^{\ast}(x) = \frac{1}{\sqrt{2 \pi}} \exp\left(-x^2 / 2\right)$, $x\in\R$. In Example \ref{ex:radon} we have already seen that $\radon{f_0}(s, u) = f^{\ast}(u-\scalarpr{x}{s})$, for $s\in\Sphere^{d-1}$, $u \in \R$. Now, we consider  $\radon{\Psi_{h, x}}$. Let $(s, {s}_1^{\perp},\ldots, {s}_{d-1}^{\perp})$ be an orthonormal basis of $\R^d$. Then $x = \scalarpr{x}{s}\cdot {s} + \sum_{j = 1}^{d-1}\langle x, {s}_j^{\perp}\rangle\cdot {s}_j^{\perp}$ and 
	\begin{align*}
		\radon{\Psi_{h, x}}(s,u)
		&= \int_{\R^{d-1}} \psi\left(h^{-1}\sqrt{\left\vert\scalarpr{x}{{s}} - u\right\vert^2 + \sum_{j=1}^{d-1}\left\vert \tau_j - \scalarpr{x}{{s}_j^{\perp}}\right\vert^2}\right) \Lebesgue{d-1}{\tau} \nonumber
	\end{align*}
 implying that $\mathrm{supp}(\radonind{s}{\Psi_{h,x}})\subseteq [\scalarpr{x}{{s}} - 1, \scalarpr{x}{{s}} + 1]$ for $s\in \Sphere^{d-1}$.
	 Thus, 
	\begin{align*}
		\chi^2\left(\mathds{P}_{\radon{f_1}}, \mathds{P}_{\radon{f_0}}\right) 
		&\leq  \frac{\delta^2 h^{2\beta-d}}{\rho_d f^{\ast}(1)} \int_{\Sphere^{d-1}} \Lpnlr{\radon{\Psi_{h, x}}(s, \cdot)}{2}{\R}^2 \surfacemeasure{d-1}{s}.
    \end{align*}
  Now the Plancherel equality and the projection theorem imply $\Lpn{\radonind{s}{\Psi_{h, x}}}{2}{\R}^2 =\Lpn{\Fourier{d}{\Psi_{h,x}}(\cdot s)}{2}{\R}^2/(2\pi)$. Therefore,
  \begin{align*}
       \Lpnlr{\radonind{{s}}{\Psi_{h, x}}}{2}{\R}^2 = \frac{h^{2d}\Lpn{\Fourier{d}{\Psi_{1,0}}(\cdot hs)}{2}{\R}^2}{2\pi} = \frac{h^{2d-1}\Lpn{\Fourier{d}{\Psi_{1,0}}(\cdot s)}{2}{\R}^2}{2\pi}.
  \end{align*}
  Following the similar arguments as in Equation \eqref{ineq:L2:radon}, we get
\begin{align*}
    \int_{\Sphere^{d-1}} \Lpn{\Fourier{d}{\Psi_{1,0}}(\cdot s)}{2}{\R}^2 \surfacemeasure{d-1}{{s}} &\leq 2\left(\rho_d \Lpn{\Psi_{1,0}}{1}{\R^d}^2 + \Lpn{\Fourier{d}{\Psi_{1,0}}}{2}{\R^d}^2\right)\\
    &\leq 2\rho_d\left(\rho_d+(2\pi)^d \right).
\end{align*} 
Therefore,
\begin{align*}
    \phantom{hallohallohalloo}\chi^2\left(\mathds{P}_{\radon{f_1}}, \mathds{P}_{\radon{f_0}}\right) \leq \delta^2 h^{2\beta+d-1}(\rho_d+(2\pi)^d)(2\pi e)^{-1/2}. \phantom{hallohallohalloo}\qedhere
\end{align*}
\end{proof}

\subsection{Proofs of Section \ref{sec_dd}}\label{a:dd}

\begin{proof}[Proof of Theorem \ref{theo:adap}]
First, for $m\in \Mn$
    \begin{equation*}
      \left|f(x)-\widehat{f}_{\widehat m}(x)\right|^2 \leq 
      2\left|f(x)- \widehat f_m(x)\right|^2 +2 \left|\widehat f_m(x)-
      \widehat f_{m \wedge \widehat m}(x)\right|^2
      + 2\left|\widehat f_{m \wedge \widehat m}(x)- \widehat f_{\widehat m}(x)\right|^2. 
    \end{equation*}
    Next, the definition \eqref{def:hVhA} of
    $\widehat A$, and secondly the definition \eqref{eq:adap:m} of
    $\widehat m$ imply
    \begin{align*}
      \left|f(x)-\widehat{f}_{\widehat m}(x)\right|^2
      &\leq
        2\left|f(x)- \widehat f_m(x)\right|^2 + 2\left( \widehat A\left(\widehat m\right)+ \widehat V(m)+ \widehat A(m)+ \widehat V\left(\widehat m\right)\right) \\
      & \leq 2\left|f(x)- \widehat f_m(x)\right|^2 + 4\left( \widehat A(m)+ \widehat V(m)\right) . 
     \end{align*}
     It is easy to check that $\widehat A(m) \leq A(m)+
     \max\{(V(m')-\widehat V(m'))_+:m'\in\Mn\}$, and hence
     \begin{align}\label{theo:adap:p:e1}
       &\phantom{\leq}|f(x)-\widehat{f}_{\widehat m}(x)|^2 \notag\\
       &\leq 2\left|\left(f- \widehat f_m\right)(x)\right|^2  + 4A(m)
       +4\left(\left(\widehat V-2V\right)(m)\right) + 8 V(m) + 4 \max_{m'\in \Mn}\left(\left(V-\widehat V\right)(m')\right)_+\notag\\
       &\leq 10V(m) + 4|f(x)- f_m(x)|^2  
       +4\left(\widehat V(m)-2V(m)\right) + 4 \max_{m'\in \Mn}\left(V\left(m'\right)-\widehat V\left(m'\right)\right)_+ \notag\\
       &\phantom{\leq}
       + 4\left(\left|f_m(x)- \widehat f_m(x)\right|^2-V(m)/3\right)_+ + 4A(m).
     \end{align}
     Further, for $m'\in  \Mn, m'>m$, we get the inequality
     \begin{align*}
     \left|\widehat f_{m'}(x)- \widehat f_{m}(x)\right|^2\leq
     3\left|\widehat f_{m'}(x)- f_{m'}(x)\right|^2+ 3\left|\widehat f_{m}(x)- f_{m}(x)\right|^2
     +3\left|f_{m}(x)- f_{m'}(x)\right|^2
     \end{align*}from which we conclude
     \begin{align}\label{theo:adap:p:e2}
       A(m) &\leq 3\max_{m'\in \llbracket m+1, \mathrm{M}_n\rrbracket} \left
         ( \left|\widehat f_{m'}(x)-f_{m'}(x)\right|^2+\left|\widehat
         f_{m}(x)-f_{m}(x)\right|^2- \frac{V(m')}{3} \right)_+ \notag
       \\&\phantom{\leq}+3\max_{m'\in \llbracket m+1, \mathrm{M}_n\rrbracket} |f_m(x)-f_{m'}(x)|^2\notag\\
       & \leq 6\max_{m'\in \llbracket m, \mathrm{M}_n\rrbracket} \left
         ( \left|\widehat f_{m'}(x)-f_{m'}(x)\right|^2 - \frac{V(m')}{3} \right)_+ +V(m)\notag\\
         &\phantom{\leq}+3\max_{m'\in \llbracket m+1, \mathrm{M}_n\rrbracket} |f_m(x)-f_{m'}(x)|^2.
     \end{align}
    Now for any $m'\in \llbracket m+1, \mathrm{M}_n\rrbracket$ 
     \begin{align}\label{theo:adap:p:e3}
       \hspace*{-0.4cm}\left|f_m(x)-f_{m'}(x)\right| \leq 
       \frac{1}{(2\pi)^d}\int_{\ball{m'}\setminus\ball{m}} \left|\Fourier{d}{f}(t)\right| \inte\lambda^d(t)
       \leq  \frac{\|\mathds{1}_{\ballc{m}}\Fourier{d}{f}\|_{\Lp{1}(\R^d)}}{(2\pi)^d}.
     \end{align}
     Combining \eqref{theo:adap:p:e1} with \eqref{theo:adap:p:e2}, \eqref{theo:adap:p:e3}, 
     $|f_m(x)-f(x)| \leq
     \|\mathds{1}_{\ballc{m}}\Fourier{d}{f}\|_{\Lp{1}(\R^d)}(2\pi)^{-d}$ and  $\EWind{\radon{f}}{\widehat
     V(m)}=2V(m)$ for each $m\in\Mn$ we obtain
     \begin{align}\label{theo:adap:p:e4}
       &\EWindlr{\radon{f}}{\left|f(x)-\widehat{f}_{\widehat m}(x)\right|^2}
       \leq 16\left(\frac{\|\mathds{1}_{\ballc{m}}\Fourier{d}{f}\|^2_{\Lp{1}(\R^d)}}{(2\pi)^{2d}} +V(m)\right) 
       \notag \\
       &+ 4 \EWindlr{\radon{f}}{\max_{m'\in \Mn}\left(V-\widehat V\right)_+\left(m'\right)}
       + 28\EWindlr{\radon{f}}{\max_{m'\in \Mn} \left( \left|\widehat f_{m'}-f_{m'}\right|^2(x)-  \frac{V(m')}{3} \right)_+}.
     \end{align}
     We bound the last two summands in \eqref{theo:adap:p:e4} with the
     help of Lemma \ref{lem:concen} below. 
     \begin{lem}\label{lem:concen}
      Under the assumptions of Theorem \ref{theo:adap} for all $\chi \geq 24$ holds
      \begin{enumerate}
          \item[$(i)$] $\EWind{\radon{f}}{\max_{m'\in \Mn}(V-\widehat V)_+(m')}\leq C(\chi, d) n^{-1}$,
          \item[$(ii)$] $\EWind{\radon{f}}{\max_{m'\in \Mn} ( |\widehat f_{m'}-f_{m'}|^2(x)-  V(m')/3)_+} \leq C(f,d) n^{-1}$ 
      \end{enumerate}
      where $C(f,d)>0$ is a positive constant depending on $f$ and $d$.
     \end{lem}
     Taking now the infimum over all $m\in \Mn$ completes the proof.
\end{proof}

\begin{proof}[Proof of Lemma \ref{lem:concen}]
\underline{Starting with $(i)$}, we have for $m'\in \Mn$
\begin{align*}
    \left(V(m')-\widehat V(m')\right)_+&= \chi \frac{(m')^{2d-1}}{n}\log(n)\left((1+\mu_{m'})-2\left(1+\widehat\mu_{m'}\right)\right)_+\\
    &\leq \chi\log(n) 2|\mu_{m'}-\widehat\mu_{m'}|\mathds 1_{\left|\widehat \mu_{m'}-\mu_{m'}\right|\geq \frac{1}{2}} \leq \chi \log(n) 2^p \left|\mu_{m'}-\widehat\mu_{m'}\right|^p
\end{align*}
for any $p\geq 1$. Now, as $|K_1(y)| \leq \rho_d/((2\pi)^dd)$ for all $y\in \R$, Theorem 2.10 of \cite{Petrov1995} implies $\EWind{\radon{f}}{|\mu_{m'}-\widehat\mu_{m'}|^p} \leq C(p, d) \left({m'}\right)^p n^{-p/2} \leq C(p, d) n^{p(3-2d)/(4d-2)}$. We deduce
\begin{align*}
    \EWindlr{\radon{f}}{\max_{m'\in \Mn}\left(V-\widehat V\right)_+(m')} &\leq 2^p\chi\log(n)\sum_{m'\in \Mn} \EWindlr{\radon{f}}{\left|\mu_{m'}-\widehat\mu_{m'}\right|^p}\\
    &\leq C(p,\chi, d) \log(n)n^{1/(2d-1)} n^{p\frac{3-2d}{4d-2}} \leq C(\chi, d)n^{-1}
\end{align*}
for $p$ large enough.\\
\underline{For $(ii)$} we use that \begin{align*}
    \max_{m^{\prime} \in \Mn} \left(\left|\widehat{f}_{m^{\prime}} \left(x\right) - f_{m^{\prime}}\left(x\right)\right|^2 - \frac{V\left(m^{\prime}\right)}{3}\right)_+  \leq \sum_{m \in \Mn} \left(\left|\widehat{f}_{m} \left(x\right) - f_{m}\left(x\right)\right|^2 - \frac{V\left(m\right)}{3}\right)_+
\end{align*}
and bound each summand separately. To do so, we exploit 
\begin{align}
   & \EWindlr{\radon{f}}{ \left(\left|\widehat{f}_{m} \left(x\right)- f_{m}\left(x\right)\right|^2 - \frac{V\left(m\right)}{3}\right)_+}\notag\\
    &\quad= \int_{0}^{\infty} \PPindlr{\radon{f}}{\left(\left(\widehat{f}_{m} \left(x\right) - \EWindlr{\radon{f}}{\widehat{f}_m\left(x\right)}\right)^2 - \frac{V\left(m\right)}{3}\right)_+ \geq y} \inte y\notag\\
    &\quad\leq \int_{0}^{\infty} \PPindlr{\radon{f}}{\left\vert \widehat{f}_m\left(x\right) - \EWindlr{\radon{f}}{\widehat{f}_m\left(x\right)} \right\vert \geq \sqrt{\frac{V(m)}{3} + y}} \inte y. 
\end{align} 
	The probability in the last term can be bounded by applying the Bernstein inequality, see Lemma \ref{ineq: bernstein}. To do so, defining $T_j := K_m\left(\scalarpr{S_j}{x} - U_j\right)$ for $j=1,\ldots,n$ yields
	\begin{align*}
		\widehat{f}_m\left(x\right) - \EWindlr{\radon{f}}{\widehat{f}_m\left(x\right)} = \frac{1}{n}\sum_{j=1}^{n} \left(T_j - \EWindlr{\radon{f}}{T_j}\right)
	\end{align*}
	and $\left\vert T_j \right\vert\leq \rho_d/((2 \pi)^d d)m^d=:b$ and $\Varind{\radon{f}}{T_j} \leq \EWind{\radon{f}}{T_j^2} \leq  m^{2d-1}(1+\mu_m)=: v^2$.  By the Bernstein inequality, see Lemma \ref{ineq: bernstein}, we have for any $y > 0$ and $\alpha \in [0,1]$
	\begin{multline*}
		\PPindlr{\radon{f}}{\left\vert \widehat{f}_m\left(x\right) - \EWindlr{\radon{f}}{\widehat{f}_m\left(x\right)} \right\vert \geq \sqrt{\frac{V(m)}{3} + y}} \\
		\leq 2 \max\left\{\exp\left(-\frac{n}{4v^2}\left(\frac{V(m)}{3} + y\right)\right), \exp\left(-\frac{n \alpha}{4b} \sqrt{\frac{V(m)}{3}}\right)\exp\left(-\frac{n(1-\alpha)}{4b} \sqrt{y}\right)\right\},
	\end{multline*}
	where the concavity of the square root has been used in the last step. By the definition of $V(m)$ and of $v^2$, we get
	\begin{align*}
		\frac{n}{4v^2}\frac{V(m)}{3} = \frac{\chi}{12} \log(n) \geq p \log(n)
	\end{align*} 
	for $\chi \geq 12 p$, where $p \in \R_+$. Moreover, 
	\begin{align*}
		\frac{n \alpha}{4b} \sqrt{\frac{V(m)}{3}} 
		&= \frac{\alpha}{4} \sqrt{\frac{\chi (2 \pi)^{2d} d^2 (1+\mu_m)}{3 \rho_d^2}} \sqrt{\frac{n \log(n)}{m}} 
		=: \gamma_{\alpha} \sqrt{\frac{n \log(n)}{m}} 
		\geq p \log(n)
	\end{align*}
	if $n/m \geq (p/\gamma_{\alpha})^2 \log(n)$. Latter holds true by choosing $\alpha \in \left[0, \frac{1}{2}\right)$ small enough so that $p /\gamma_{\alpha} > 1$ is satisfied. Then the above inequality holds under the definition of $\Mn$. From this we deduce that for $p=2$ and as $\mu_m \leq C_d(1+ \|\Fourier{d}{f}\|_{\Lp{1}(\R^d)})$ for some positive constants $C_d$ only depending on $d$ that
    \begin{align*}
        \PPindlr{\radon{f}}{\left\vert \widehat{f}_m\left(x\right) - \EWindlr{f}{\widehat{f}_m\left(x\right)} \right\vert \geq \sqrt{\frac{V(m)}{3} + y}} 
        &\leq \frac{2}{n^2}  \max\left\{e^{-\frac{ny}{4v^2}}, e^{- \frac{n}{8b}\sqrt{y}}\right\}  \\
        &\leq  \frac{2}{n^2}\left(e^{- \tau_1 y}\vee e^{-\tau_2  \sqrt{y}}\right)
    \end{align*}
    where $\tau_1 := 1 /(4 C_d \left(2 + \Lpn{\Fourier{d}{f}}{1}{\R^d}\right))$, $\tau_2 := ((2 \pi)^d d)/(8 \rho_d)$. Thus 
    \begin{align*}
		\EWindlr{\radon{f}}{\max_{m^{\prime} \in \Mn}\left(\left|\widehat{f}_{m^{\prime}}\left(x\right) - f_{m^{\prime}}\left(x\right)\right|^2 - \frac{V\left(m^{\prime}\right)}{3}\right)_+} &\leq \frac{2}{n^2} \sum_{m \in \Mn} \int_0^{\infty} \left(e^{- \tau_1 y}\vee e^{-\tau_2  \sqrt{y}}\right) \inte y\\
		&\leq \frac{2}{n^2} \sum_{m \in \Mn} \left(\frac{1}{\tau_1}\vee \frac{2}{\tau_2^2}\right)
		\leq \frac{C(f, d)}{n} 
	\end{align*}
	using $m \leq n^{\frac{1}{2d-1}} \leq n^{\frac{1}{d}}$ and the definition of $\Mn$.
\end{proof}

\begin{proof}[Proof of Theorem \ref{thm:low:adap}]
    We use the same construction and similar steps as in the proof of Theorem \ref{thm:low:bou}. 
    For $f_0$, we choose the density of the multivariate normal distribution, i.e.
	\begin{align}\label{eq:f:2}
		f_0(y) &= \frac{1}{\left(\sqrt{2 \pi}\right)^d} \exp\left(- \frac{1}{2} \scalarpr{x-y}{x-y}\right) \text{ for } x, y \in \R^d.
  \end{align}
	 Let $\psi \in C_0^\infty\left(\R\right)$ be a function with  $\mathrm{supp}(\psi)\subseteq[0,1]$, $\psi(0)=1$, $\|\psi\|_{\infty}=1$ and $\int_0^1 y^{d-1} \psi(y) \inte y = 0$. For a bump-amplitude $\delta > 0$ and for $h \in (0,1)$ (to be chosen later on) we define
    \begin{align}
		f_1(y) &:= f_0(y) + \delta h^{\beta' - \frac{d}{2}} \Psi_{h, x}(y), \label{eq: f1 fourier radon adap}
	\end{align}
	where $x$, $y \in \R^d$ and $\Psi_{h, x}(y) := \psi(\| y - x \|/h)$. Let $0 < \delta \leq (2\pi e)^{-d/2}$. Then $f_1$, as in Equation \eqref{eq:f:2}, is a density. Similar, to the proof of Lemma \ref{hilfslemma 2: f0 f1 in fourier sobolev radon pointwise} we see that $f_0 \in \Sobolev{\beta}{L}$, for $L > L_{\beta, d}$ and $f_1 \in \Sobolev{\beta'}{L}$ for $L > L_{\beta', d}$ where $L_{\beta',d}< L_{\beta, d}$. More precisely, for $L>L_{\beta, d}$ we have
 \begin{enumerate}
    \item[a)] $f_0$ and $f_1$ are densities for $\delta<(2\pi e)^{-d/2}$,
    \item[b)] $f_0\in \Sobolev{\beta}{L}$ and $f_1 \in \Sobolev{\beta'}{L}$,
    \item[c)] $|f_0(x)-f_1(x)| = \delta h^{\beta'-d/2}$ and
    \item[d)] $\chi^2\left(\mathds{P}_{\radon{f_1}},\mathds{P}_{\radon{f_0}}\right) \leq C(d)\delta^2 h^{2\beta'+d-1}$ with $C(d)= (\rho_d+(2\pi)^d)(2\pi e)^{-1/2}$.
 \end{enumerate}
 Now by assumption, and as $f_0 \in \Sobolev{\beta}{L}$ we have $$\EWindlr{\radon{f_0}}{\left|\widehat f_n(x) - f_0(x)\right|^2} \leq \mathfrak C n^{-\frac{2\beta-d}{2\beta+d-1}}.$$ First, we use the Jensen inequality to deduce $\sup_{f \in \Sobolev{\beta'}{L}}(\EWind{\radon{f}}{\vert \widehat f_n(x) - f(x) \vert^2} )^{1/2} 
        \geq \EWind{\radon{f_1}}{\vert \widehat f_n(x) - f_1(x)\vert}$. This gives, using the Radon-Nikodym derivative and Cauchy-Schwarz inequality, 
       \begin{align*}
        &\EWindlr{\radon{f_1}}{\left\vert \widehat f_n(x) - f_1(x)\right\vert} \geq |f_0(x)-f_1(x)| -  \EWindlr{\radon{f_1}}{\left\vert \widehat f_n(x) - f_0(x)\right\vert}\\
        &\geq \left\vert f_0(x) - f_1(x) \right\vert - \EWindlr{\radon{f_0}}{\left\vert \widehat f_n(x) - f_0(x)\right\vert^2}^{\frac{1}{2}} \left(\EWindlr{\radon{f_0}}{\left(\frac{\inte \mathds{P}^n_{\radon{f_1}}}{\inte \mathds{P}^n_{\radon{f_0}}} - 1\right)^2}^{\frac{1}{2}} + 1\right) \\
        &\geq \delta h^{\beta'- d/2} - 2\mathfrak C^{1/2} n^{-\frac{2\beta - d}{4\beta + 2d -2}} \left(1 + \chi^2\left(\mathds{P}^n_{\radon{f_1}}, \mathds{P}^n_{\radon{f_0}}\right)\right), 
    \end{align*}
    According to \cite{Tsybakov2008}, we have 
    \begin{align*}
        1+\chi^2\left(\mathds{P}^n_{\radon{f_1}}, \mathds{P}^n_{\radon{f_0}}\right) &= \prod_{j=1}^n \left( 1 + \chi^2 \left(\mathds{P}_{\radon{f_1}}, \mathds{P}_{\radon{f_0}}\right)\right)  
        \leq \exp\left(\sum_{j=1}^n \chi^2\left(\mathds{P}_{\radon{f_1}}, \mathds{P}_{\radon{f_0}}\right)\right). 
    \intertext{Now let $\xi:= \frac{2\beta-d}{2\beta+d-1} - \frac{2\beta'-d}{2\beta'+d-1}>0$. Using d) yields}
        1+ \chi^2\left(\mathds{P}^n_{\radon{f_1}}, \mathds{P}^n_{\radon{f_0}}\right) &\leq \exp\left( C(d) n\delta^2 h^{2\beta' + d -1}\right) \leq \exp\left(\xi nh^{2\beta'+d-1}/2\right)
    \end{align*}
    for $\delta < \delta(\beta, \beta', d)$.  Choosing $h=(\frac{n}{\log(n)})^{-1/(2\beta'+d-1)}$ leads to
    \begin{align*}
    \sup_{f \in \Sobolev{\beta'}{L}}\left(\EWindlr{\radon{f}}{\left\vert \widehat f_n(x) - f(x) \right\vert^2} \right)^{\frac{1}{2}} & \geq \left(\frac{n}{\log(n)}\right)^{-\frac{\beta'-d/2}{2\beta'+d-1}}\left(\delta-2\mathfrak C^{1/2}\log(n)^{-\frac{\beta'-d/2}{2\beta'+d-1}} \right)\\
    &\geq\mathfrak c(\mathfrak C, \beta, \beta', d)\left(\frac{n}{\log(n)}\right)^{-\frac{\beta'-d/2}{2\beta'+d-1}}
    \end{align*}
    for $n\in \N$ big enough. This completes the proof.
\end{proof}

\bibliographystyle{plain} 
\bibliography{reference}

\end{document}